\documentclass{amsart}
\usepackage{amsmath,amssymb,amsthm,amscd, xypic}
\usepackage{color}
\usepackage{enumerate}
\usepackage[T1]{fontenc} 
\usepackage{hyperref}

\title[Brown-Adams representability]{Brown--Adams representability for triangulated categories with locally coherent cohomology}

\author{George Ciprian Modoi}
\address[George Ciprian Modoi]{Babe\c s--Bolyai University, Faculty of Mathematics and Computer Science \\
1, Mihail Kog\u alniceanu, 400084 Cluj--Napoca, Romania}
\email[George Ciprian Modoi]{cmodoi@math.ubbcluj.ro}

\renewcommand{\iff}{if and only if }

\newcommand{\la}{\longrightarrow}

\newcommand{\Z}{\mathbb{Z}}

\DeclareMathOperator{\Hom}{Hom}

\DeclareMathOperator{\End}{End}
 
\DeclareMathOperator{\Ker}{Ker}

\DeclareMathOperator{\colim}{\mathrm{colim}}
\DeclareMathOperator{\Colim}{\mathrm{Colim}}

\newcommand{\A}{\mathcal{A}}
\newcommand{\B}{\mathcal{B}}
\newcommand{\I}{\mathcal{I}}

\newcommand{\C}{\mathcal{C}}
\newcommand{\D}{\mathcal{D}}

\newcommand{\G}{\mathcal{G}}
\newcommand{\CS}{\mathcal{S}}

\newcommand{\ModK}{\mathrm{Mod}\textrm{-}R}
\newcommand{\ModR}{\mathrm{Mod}\textrm{-}R}

\newcommand{\modK}{\mathrm{mod}\textrm{-}R}
\newcommand{\md}[1]{\hbox{\rm mod-}{#1}}

\newcommand{\Qcoh}{\hbox{\rm Qcoh-}}

\newcommand{\op}{^\mathrm{op}}
\newcommand{\Mod}[1]{\hbox{\rm Mod-}{#1}}
\newcommand{\add}{\mathrm{add}}
\newcommand{\Add}{\mathrm{Add}}

\newcommand{\Spec}[1]{\mathrm{Spec}({#1})}

\newcommand{\Der}[1]{\mathbf{D}({#1})}
\newcommand{\Dec}[1]{\mathbf{D}_{coh}({#1})}
\newcommand{\Decm}[1]{\mathbf{D}^{-}_{coh}({#1})}
\newcommand{\Derq}[1]{\mathbf{D}_{Qcoh}({#1})}
\newcommand{\Dep}[1]{\mathbf{D}_{perf}({#1})}
\newcommand{\Derb}[1]{\mathbf{D}^{\mathrm b}_{coh}({#1})}

\newcommand{\thick}{\mathrm{thick}\,}

\newcommand{\smd}{\mathrm{smd}\,}

\newcommand{\hocolim}{\underrightarrow{\mathrm{hocolim}}}
\newcommand{\yon}{\mathbf{y}}
\newcommand{\ppdim}{\mathrm{pp.dim}}
\newcommand{\pdim}{\mathrm{p.dim}}

\theoremstyle{plain}
\newtheorem{thm}{Theorem}[section]
\newtheorem{lemma}[thm]{Lemma}
\newtheorem{prop}[thm]{Proposition}
\newtheorem{cor}[thm]{Corollary}

\newtheorem{thm-intro}{Theorem}

\theoremstyle{definition}

\theoremstyle{remark}
\newtheorem{rem}[thm]{Remark}

\begin{document}


\begin{abstract}
In this paper, we deal with two types of representability. The first is a variant of the Brown representability theorem in the spirit of Rouquier \cite{R08} and Neeman \cite{N18}. The second is a variant of the Brown-Adams representability. If $A$ is a dg-algebra over a commutative noetherian ring $R$, such that $A$ has coherent cohomology, it is shown that every cohomological (contravariant) functor $M:\Dep A\to\ModR$, also satisfying $M(A[-n])\in\modK$, for all $n\in\Z$ is isomorphic to $\Der A(-,X)|_{\Dep A}$, where $X\in\Der A$ is such that $H^n(X)$ is coherent for all $n\in\Z$. 

\end{abstract}

\maketitle

\section*{Introduction}

The importance of triangulated categories in modern mathematics is difficult to overestimate. However, triangulated categories have some well-known drawbacks. One of them is that they are not (co)complete and Freyd's adjoint functor theorem does not apply. Brown representability comes as a substitute since in the presence of this property, triangulated functors commuting with (co)products have adjoints.   
Let $\D$ be a triangulated $R$-linear category, where $R$ is a base commutative ring. The theory of Brown representability looks for conditions (necessary and sufficient) for a contravariant functor $\D\to\ModR$ to be representable. Initially, this was done by Brown itself for the homotopy category of spectra, \cite{B62}; this involved contravariant functors sending coproducts into products and was generalized for compactly, and even for well generated triangulated categories. The common feature of these categories is that they that can be constructed in some way from a small subcategory (see \cite[Chapter 8]{N01}). Initially, the objects in the small category were supposed to be nice enough, but later it became clear that more important was the fact that they are only a set; see \cite{N09}. A different, but somehow related problem is the so-called Brown-Adams representability. Now we consider a subcategory $\C$ and the restricted Yoneda functor $\yon X=\D(-,X)|_\C$, mapping an object $X\in\D$ to a contravariant functor $\C\to\ModR$. The problem is to identify functors $\C\to\ModR$ and natural transformations between them that are of the form $\yon X$, for some $X\in\D$, respectively $\yon(f)$, for some $f:X\to Y$ in $\D$. If $\C=\D$, then we recover the Brown representability above, where representability is taken in the categorical sense. But the prototypical example is  when $\D$ is compactly generated and $\C=\D^c$ is the category of compact objects. For this case, we know that every cohomological functor $\C\to\ModR$ is of the form $\yon X$, for some $X\in\D$ and every natural transformation between cohomological functors comes from a map in $\D$ (that is, $\yon$ is full and its essential image consists of all cohomological functors) \iff pure projective dimension of $\D$ is $\leq 1$; see \cite{CS98}.

In this paper, we consider an $R$-linear category $\D$. All subcategories will be full and all functors will be $R$-linear. 
The first Section gives a very general criterion of representability, see Corollary \ref{c:rep-iff-fp}. Here we do not need the entire triangulated structure on $\D$, the existence of weak (co)kernels being enough. In this context, we also find a criterion for a subcategory to be coreflective; see Corollary \ref{c:tri+precover}.

Brown representability for categories without (co)products seems to be a more intricate problem than the classical one. To the best knowledge of the author, there are only three instances of results of this kind in the literature, namely \cite[Theorem 1.3]{BvdB03}, \cite[Theorem 4.20]{R08} and \cite[Theorem 4.2]{N18}. In Section 2 we combine ideas from the last two to give a new version which is intended to be as general as possible. Theorem \ref{thm:main} contains this result. Essentially, the idea is the one already mentioned: all objects of the category have to be constructed in a finite number of steps from a small set by taking direct summands, shifts, and triangles. The method also works for categories with coproducts: Allowing not only summands, shifts, and triangles but also coproducts,  we get Brown representability for "big" categories as in \cite{N09}; see Corollary \ref{c:rep-4-big}.  

Section 3 contains a relativization of the results concerning the relationship between the fullness of the restricted Yoneda functor $\yon:\D\to\Mod{\D^c}$ and the pure projective dimension of objects in $\D$. 

In Section 4 we assume that $\D$ has a single compact generator $G$, and we show that the restriction of $\yon$ to an appropriate subcategory $\D_{lcoh}$ of $\D$ is full. Moreover, we also identify the essential image of this restriction; see Theorem \ref{t:BA}. The idea of the proof comes from a combination between \cite[Section 7]{N18} and the so-called "two-sided approximation" of \cite{B22}. Note that under suitable hypotheses, namely $R$ is noetherian and $\D(G,G[n])$ finitely generated for all $n$, the category $\D_{lcoh}$ contains exactly those objects $X\in\D$ for which $\D(C,X)$ is finitely generated for all $C\in\D^c$, and the essential image of the functor $\yon:\D_{lcoh}\to\ModR$ consists exactly from those functors (modules) $M:\D^c\to\ModR$ satisfying the property that $M(C)$ is finitely generated, for all $C\in\D^c$. 

Section 5 contains an application of our results, where the category $\D=\Der A$ is the derived category of a dg-algebra $A$. It is clear that our results are more general than the corresponding results in \cite{N18} and \cite{B22}. Indeed, these papers assume (weakly) approximability to show Brown-Adams representability (i. e., that the restriction of the functor $\yon$ is full). This assumption is not necessary in Proposition \ref{p:BA4dg}.  Here $R$ is supposed to be noetherian and $A$ is a dg-algebra with coherent cohomology. Then $\D_{lcoh}$ coincides with $\Dec A$, the subcategory of $\Der A$ consisting of all $X\in\Der A$ having coherent cohomology.

\section{A general representability criterion}

Throughout this section $R$ is a commutative ring with one and $\D$ is an additive $R$-linear category. We denote by $\ModR$ the category of all $R$-modules and by $\modK$ the category of finitely presented $R$-modules. Let $M:\D\to\ModR$ be a contravariant functor. It is customary to regard the category $\D$ as being a ring with several objects, and further the functor $M$ as being a right $\D$-module. Remark that the hypothesis on $\D$ implies that it has finite biproducts. We call $\D$ {\em karoubian}, if every idempotent endomorphism in $\D$ splits. In what follows, we look for some necessary and sufficient conditions for $M$ to be representable, that is, to be of the form $\D(-,X)$ for some $X\in\D$. The functor $M$  is called {\em finitely generated} if there is a natural epimorphism $\D(-,D)\to M$, for some $D\in\D$. 
Dually we say that $M$ is a subfunctor of a representable functor if there is a natural monomorphism $M\to\D(-,D)$. 
Furthermore, the functor $M$ is called {\em finitely presented} if there is an exact sequence
\[\D(-,C)\to\D(-,D)\to M\to 0.\] It is routine to show that $M$ is finitely presented if and only if it is finitely generated and for every $D\in\D$, the kernel of any natural epimorphism $\D(-,D)\to M$ is also finitely generated. 

For two contravariant functors $M,N:\D\to\ModR$, denote by $\Hom_\D(M,N)$ the class of all natural transformations between them.
We denote by $\Mod\D$ the class of all contravariant functors $\D\to\ModR$; as we noticed, we view them as $\D$-modules. In order to organize $\Mod\D$ as a category with the morphisms $\Hom_\D(M,N)$ we usually have to enlarge the universe. However, this set-theoretic issue does not appear in the following two situations:
If $\D$ is essentially small, then $\Hom_\D(M,N)$ is actually a set; therefore $\Mod\D$ belongs to the same universe as $\D$.
If $M$ is finitely presented, it follows by the Yoneda lemma that $\Hom_\D(M,N)$ is a set. Therefore we can construct, without enlarging the universe, the category whose objects are all finitely presented contaravariant functors $\D\to\ModR$ and whose morphisms are natural transformations. We denote by $\md\D$ this category and we call it the category of finitely presented (right) $\D$-modules.

For the rest of this section, fix a $\D$-module $M$.
Consider now a (full) subcategory $\A\subseteq\D$. We say that $M$ is {\em$\A$-finitely generated} if there is $A\in\A$ and a natural transformation $\alpha:\D(-,A)\to M$ such that \[\alpha^X:\D(X,A)\to M(X)\to 0\] is exact for all $X\in\A$. Remark that $M$ is finitely generated exactly if it is $\D$-finitely generated and the two definitions are consistent, in the sense that $M$ is $\A$-finitely generated if and only if its restriction $M|_\A$ is finitely generated. Indeed by Yoneda lemma both natural transformations $\D(-,A)\to M$, respectively $\A(-,A)\to M|_\A$ are in correspondence with an element in $M(A)$, therefore given one we can construct the other one.

Assume now that $\D$ has {\em weak kernels}, that is, given $Y\to Z$ in $\D$, there is $X\to Y$ such that the sequence of functors \[\D(-,X)\to\D(-,Y)\to\D(-,Z)\] is exact. In this case, it is clear that the composite morophism $X\to Y\to Z$ vanishes, and we call it a {\em weak kernel sequence}. The notions {\em weak cokernel} and {\em weak cokernel sequence} are defined dually. By \cite[Theorem 1.4]{F66} we know that the existence of weak kernels is equivalent to the fact that $\md\D$ is abelian (see also \cite[Lemma 2.2]{K98}). Moreover, projective objects in $\md\D$ are direct summands of representable functors, see \cite[Proposition 2.3]{K98}. Consider further an abelian category $\A$, and a covariant functor $F:\D\to\A$. By \cite[Universal Property 2.1]{K98} there is a unique, up to a natural isomorphism, right exact (covariant) functor $\widehat F:\md\D\to\A$ such that $\widehat F\D(-,X)=F(X)$ for all $X\in\D$. By a combination of \cite[Lemma 2.5 and Proposition 2.3]{K98}, we infer that the functor $\widehat F$ is exact if and only if $F$ is {\em weak left exact}, that is, it sends every weak kernel sequence $X\to Y\to Z$ to an exact sequence $F(X)\to F(Y)\to F(Z)$ in $\A$. Remark that it is enough to check the weak left exactness for any particular choice of a weak kernel for $Y\to Z$. In fact, if $X\to Y\to Z$ and $X'\to Y\to Z$ are two weak kernel sequences, then the maps $X\to Y$ and $X'\to Y$ factor through each other, therefore the images of the maps $F(X)\to F(Y)$ and $F(X')\to F(Y)$ coincide.
If $F:\D\to\A$ is now a contravariant functor, we can apply the above statements for it, seen as a covariant functor $\D\to\A\op$ and we get a unique contravariant functor $\widehat F:\md\D\to\A$ such that $\widehat F\D(-,X)=F(X)$ for all $X\in\D$ and $\widehat F$ sends cokernels to kernels.  Specializing in the discussion above in the case where the codomain category is $\A=\ModR$, we obtain a generalization (and a completion) of \cite[Lemma 1]{M13}:

\begin{prop}\label{p:hatH-rep} Assume $\D$ has weak kernels and let $M:\D\to\ModR$ be a contravariant functor. 
\begin{enumerate}[{\rm a)}]
    \item There is a natural isomorphism $\widehat M\cong\Hom_\D(-,M)$.
In particular, $M$ is finitely presented if and only if $\widehat M$ is representable.
    \item If, in addition, $M$ is a subfunctor of a representable functor, then every kernel of a natural transformation $\D(-,C)\to M$ is finitely generated. Consequently $M$ is finitely presented \iff it is finitely generated. 
\end{enumerate}
\end{prop}

\begin{proof}
 a) The isomorphism follows by the facts that $\Hom_D(-,M)$ sends cokernels to kernels, and $\Hom_\D(\D(-,X),M)\cong M(X)$, according to the Yoneda lemma. The last statement is now obvious, $\widehat M$ being represented by $M$; note that the notation $\Hom_\D(-,M)$ makes sense even for $M\notin\md\D$.

 b) From the hypothesis there is a natural monomorphism $M\to\D(-,D)$, for some $D\in\D$. If $\alpha:\D(-,C)\to M$ is a natural transformation with the kernel $K$, then the composed map $\D(-,C)\stackrel{\alpha}\to M\to\D(-,D)$ is of the form $g^*:\D(-,C)\to\D(-,D)$, for some $g:C\to D$, by Yoneda. It follows $K=\Ker\alpha=\Ker(g^*)$. Complete $g$ to a weak kernel sequence $B\stackrel{f}\to C\stackrel{g}\to D$. The diagram with exact rows
 \[\diagram &\D(-,B)\rto^{f^*}&\D(-,C)\ddouble\rto^{g^*}&\D(-,D)\ddouble\\
        0\rto&K\rto&\D(-,C)\rto^{g^*}&\D(-,D)\enddiagram\] can be completed commutatively with a map $\D(-,B)\to K$, which must be a natural epimorphism. 
\end{proof}

Consider again a contravariant functor $F:\D\to\A$. In this case, the exactness of $\widehat F$ is clearly equivalent to the exactness of the sequence $F(Z)\to F(Y)\to F(X)$ for every weak kernel sequence $X\to Y\to Z$ in $\D$. We call {\em weak right exact} a contravariant functor $F$ satisfying this last property. Note that when we use this terminology we see the contravariant functor $F:\D\to\A$ rather as a covariant functor $\D\op\to\A$; this is the reason we refer to contravariant functors and we do not identify them neither with covariant functors $\D\op\to\A$ nor with functors $\D\to\A\op$ because we want to see them in one way or in another depending on the context.

An abelian category is said to be {\em Frobenius}, provided that projective and injective objects coincide. 

\begin{cor}\label{c:rep-iff-fp}
 Assume $\D$ has weak kernels and that $\md\D$ is Frobenius. Let $M:\D\to\ModR$ be a contravariant functor. If $M$ is weak right exact, then $M$ is isomorphic to a direct summand of a representable functor \iff $M$ is finitely presented. If, in addition, $\D$ is karoubian, then $M$ is representable \iff $M$ is finitely presented.
\end{cor}

\begin{proof}
A representable functor is obviously finitely presented, therefore so is any direct summand of it. Conversely, if $M$ is finitely presented, that is, $M\in\md\D$, it follows by Proposition \ref{p:hatH-rep} a) that $\widehat M\cong\Hom_\D(-,M)$ is representable. But $M$ is weak right exact, and therefore $\widehat M$ has to be exact. This means $M$ is an injective object in $\md\D$. By hypothesis, injective objects in $\md\D$ coincide with projective objects and are direct summands of representable functors.
\end{proof}

Let $M:\D\to\ModR$ be a contravariant functor. If $\D$ is essentially small, then we know that $M$ is the colimit of all representable functors mapping to it. More precisely, we consider the category $\D/M$ (which is also essentially small) whose objects are pairs of the form $(D,x)$ with $D\in\D$ and $x\in M(D)$, and the morphisms between $(C,y)$ and $(D,x)$ are those maps $f:C\to D$ in $\D$ for which $M(f)(x)=y$.  Then $M\cong\colim_{(D,x)\in\D/M}\D(-,D)$, that is, $M$ is the colimit of the functor $\D/M\to\Mod\D, (D,x)\mapsto\D(-,D)$, where the canonical maps of the colimit are the natural transformations $\xi:\D(-,D)\to M$ associated to $x\in M(D)$ via the Yoneda isomorphism.  If we want to express the same result for categories which are not essentially small, one option is to enlarge the universe. But the existence of some colimits and adjoints is sensitive with respect to this enlargement, therefore, in the following we will make another choice, namely we remain in the same universe, but we consider big colimits only locally. More precisely, let $\I$ be a category. An $\I$-shaped diagram of contravariant functors is an assignment $i\mapsto M_i$, where $M_i:\D\to\ModR$
for any object $i$ of $\I$ and $\alpha\mapsto M_\alpha$, where $M_\alpha$ is a natural transformation $M_i\to M_j$, for every map $\alpha:i\to j$ in $\I$, satisfying the obvious functorial properties. We say that the contravariant functor $M:\D\to\ModR$ is the large colimit of $(M_i)_{i\in\I}$, and we denote $M=\Colim_{i\in\I}M_i$ if there are natural transformations $\xi_i:M_i\to M$ such that the cocone $\left(\xi_i^X:M_i(X)\to M(X)\right)_{i\in\I}$ is a colimit cocone in $\ModR$ for all $X$ in $\D$, that is \[M(X)\cong\colim_{i\in\I}M_i(X),\hbox{ for all }X\in\D.\] Note that the last colimit is genuine in $\ModR$ even if it is indexed over a possibly large class. It is clear that, if they exist,  colimits indexed over the large class $\I$ are exact in $\ModR$, provided that $\I$ is directed; consequently, large directed colimits are component-wise exact. Furthermore, the Yoneda lemma implies that if $M\cong\Colim_{i\in\I}M_i$ as before and $X\in\D$ then $\Hom_\D(\D(-,X),M)\cong\colim_{i\in\I}M_i(X)$. Therefore if we suppose further that $\I$ is directed, and $N$ is finitely presented, the exactness of $\I$-indexed colimits leads to existence of an isomorphism \[\Hom_\D(N,M)\cong\colim_{i\in\I}\Hom_\D(N,M_i).\] 
An example of a big colimit as before is obtained by observing that for any category $\D$ we have \[M\cong\Colim_{(D,x)\in\D/M}\D(-,D).\] 
We call {\em flat} a contravariant functor $M:\D\to\ModR$ with the property that there is a directed category $\I$ and an $\I$-shaped diagram $\left(\D(-,D_i)\right)_{i\in\I}$ consisting of representable functors such that \[M\cong\Colim_{i\in\I}\D(-,D_i).\]

\begin{prop}\label{p:wlf}  
Assume $\D$ is an additive $R$-category with weak kernels and weak cokernels. Let $M:\D\to\ModR$ be a contravariant functor. Then 
\begin{enumerate}[{\rm a)}]
\item $M$ is flat if and only if $M$ is weak left exact.
\item If $M$ is weak left exact and finitely presented, then $M$ is a direct summand of a representable functor.
\end{enumerate}
\end{prop}

\begin{proof}
a). If M is flat, then $M\cong\Colim_{i\in\I}\D(-,D_i)$, with $D_i\in\D$,  where $\I$ is a directed category. If $X\to Y\to Z$ is a cokernel sequence, then
\[\D(Z,D_i)\to\D(Y,D_i)\to\D(X,D_i)\] is exact by the definition of the weak cokernel, and since directed colimits are exact in $\ModR$ we get the weak left exactness of $M$. Conversely, we noted before that  \[M\cong\Colim_{(D,x)\in\D/M}\D(-,D).\]  We only have to show that if $M$ is weak left exact, then $\D/M$ is directed. It is enough to show that it has coproducts and coequalizers. The existence of coproducts is obvious. Furthermore, if $f,g:(C,y)\to(D,x)$ are two parallel maps in $\D/M$, that is, $f,g:C\to D$ are genuine maps in $\D$ such that $M(f)(x)=y=M(g)(x)$, then choosing a weak cokernel for $f-g$ we get a weak cokernel sequence $C\stackrel{f-g}\la D\stackrel{h}\la E$. The exactness of the sequence $M(E)\to M(D)\to M(C)$ gives us an element $z\in M(E)$ such that $M(h)(z)=x$, which shows that $\D/M$ has coequalizers, proving our claim. 

b) Since $M$ is weak left exact, it is flat as we have seen above, therefore $M\cong\Colim_{i\in\I}\D(-,D_i)$, with $D_i\in\D$,  where $\I$ is a directed category. If, in addition, $M$ is finitely presented, then 
\[\Hom_\D(M,M)\cong\colim_{i\in\I}\Hom_\D(M,\D(-,D_i)).\] This implies that $1_M$ is the image of some map in $\Hom_\D(M,\D(-,D_i))$, that is, it factors as $M\to\D(-,D_i)\to M$, for some $i\in\I$. This means $M$ is a direct summand of a representable functor.

\end{proof}

 In order to formulate the next result, we need to introduce some terminology. If $\CS$ is a (full) subcategory of $\D$ and recall that an {\em $\CS$-precover} of $D\in\D$ is a morphism $S\to D$ such that the induced map $\D(X,S)\to\D(X,D)$ is surjective for all $X\in\CS$. The subcategory $\CS$ is called {\em precovering}, provided that every $D\in\D$ has an $\CS$-precover.

\begin{cor}\label{c:tri+precover} Assume $\D$ is a karoubian additive $R$-category with weak kernels and weak cokernels. Let  $\CS\subseteq\D$ be a full subcategory which is closed under weak cokernels and direct summands. If $\CS$ is precovering, then $\CS$ is coreflective.
\end{cor}

\begin{proof}
Let $D\in\D$. Consider an $\CS$-precover $S\to D$ and complete it to a weak kernel sequence $C\to S\to D$. Consider further an $\CS$-precover $T\to C$. Together, these imply that the sequence
\[\CS(X,T)\to\CS(X,S)\to\D(X,D)\to0\] is exact for all
$X\in\CS$. This means that the cohomological functor $\D(-,D)|_{\CS}:\CS\to\ModR$ is finitely presented. By hypothesis $\CS$ is karoubian and closed under weak cokernels. That means $\CS$ itself has weak cokernels, and $\D(-,D)|_{\CS}$ sends them to exact sequences. Moreover, if $S\to T$ is a map in $\CS$, then taking first a weak kernel in $\D$, namely $K\to S\to T$ and then a $\CS$-precover $R\to D$ we get a weak kernel sequence $R\to S\to T$ in $\CS$. All the hypotheses of Proposition \ref{p:wlf} b) relative to the functor $\D(-,D)|_{\CS}$ are satisfied, so it is representable. Finding an object $S_D$ that represents it is equivalent to defining a right adjoint $D\mapsto S_D$ to the inclusion functor $\CS\to\D$.
\end{proof}

\begin{rem} Let specialize us to the when case $\D$ is a triangulated category. With this assumption the completion of a map $Y\to Z$ (respectively $X\to Y$) to a triangle $X\to Y\to Z\rightsquigarrow$ provides a weak (co)kernel, therefore a functor $F:\D\to\A$ is weak left exact if and only if it weak right exact and if and only if is homological. Moreover we know by \cite[Corollary 5.1.23]{N01}, that $\md\D$ is Frobenius. A subcategory $\CS$ is called {\em suspended} if it is closed under extensions and positive shifts. Therefore if $\CS$ is a suspended subcategory, then it is closed under weak cokernels in $\D$ and all hypotheses of Corollary \ref{c:tri+precover} are satisfied, leading  to the following generalization for \cite[Proposition 1.4]{N10}: Every suspended karoubian subcategory $\CS$ of $\D$ is coreflective. Note also that despite the fact that in \cite[Proposition 1.4]{N10}, the subcategory $\CS$ is supposed to be triangulated, therefore we may apply Corollary  \ref{c:rep-iff-fp} in order to get the same conclusion, the original proof is much closer to the first one which uses Corollary \ref{c:tri+precover}. The fact that the proof in \cite{N10} can be applied for subcategories which are only suspended (and not necessary triangulated) was also observed in \cite[Theorem 2.9]{LV20}. We can compare Corollary \ref{c:tri+precover} with the results appeared in \cite[Theorem 2.2 and Corollary 2.3]{CGR14}: The same conclusion that $\CS$ is coreflective it is proved under slightly different hypotheses. 
\end{rem}

\section{Representability for functors defined on triangulated categories}

In this section, the category $\D$ is triangulated. As we noted before, $\md\D$ is Frobenius, and a functor $F:\D\to\A$ is weak left exact if and only if it weak right exact and if and only if it is homological.

For two subcategories $\A,\B\subseteq\D$ we define $\smd\A$ as the closure of $\A$ under direct summands and
\[\A*\B=\{Y\in\D\mid\hbox{ there is a triangle }X\to Y\to Z\rightsquigarrow\hbox{ with }X\in\A, Z\in B\}.\]
For $n\geq1$ we define inductively $\A^1=\smd\A$ and $\A^{*(n+1)}=\smd(\A^{*n}*\A^1)$. Since $0\in\A^{*1}$ we deduce $\A^{*1}\subseteq\A^{*2}\subseteq\ldots$.

\begin{lemma}\label{l:HisA*n-fg}
Let $M:\D\to\ModR$ be a cohomological functor, and let $\A\subseteq\D$ be a (full) subcategory such that $\A[1]=\A$.  Suppose that $M$ is $\A$-finitely generated and, for any $D\in\D$, the kernel of every natural transformation $\D(-,D)\to M$ is also $\A$-finitely generated. Then $M$ is $\A^{*n}$-finitely generated for all $n\geq1$.
\end{lemma}

\begin{proof}
 We proceed by induction on $n$. For $n=1$, $M$ is $\A$-finitely generated by hypothesis, so it is also $\A^{*1}$-finitely generated. Suppose that $M$ is $A^{*n}$-finitely generated for some $n\geq1$. Then there are $A_n\in\A^{*n}$ and $\alpha_n:\D(-,A_n)\to M$ such that $\alpha_n^X$ is surjective for all $X\in\A^{*n}$. Denote by $K_n$ the kernel of $\alpha_n$. By hypothesis $K_n$ is $\A$-finitely generated, so there is $B_n\in\A$ and a natural transformation $\D(-,B_n)\to K_n$ such that $\D(X,B_n)\to K_n(X)\to 0$ is exact for all $X\in\A^{*1}$. By the Yoneda lemma, the composite $\D(-,B_n)\to K_n\to\D(-,A_n)$ is of the form $b_n^*=\D(-,b_n)$ for some (unique) $b_n:B_n\to A_n$ in $\D$. Complete it to a triangle \[B_n\stackrel{b_n}\la A_n\stackrel{a_n}\la A_{n+1}\stackrel{c_n}\la B_n[1].\] By construction $\alpha_nb_n^*=0$, which means that $M(b_n)$ maps to $0$ the element in $M(A_n)$ corresponding under Yoneda isomorphism to $\alpha_n$. Since $M$ is cohomological, the sequence \[M(A_{n+1})\stackrel{M(a_n)}\la M(A_n)\stackrel{M(b_n)}\la M(B_n)\] is exact, therefore, this element in $M(A_n)$ lies in the image of $M(a_n)$. Applying the Yoneda lemma again, this means that there is a natural transformation $\alpha_{n+1}:\D(-,A_{n+1})\to M$ such that $\alpha_{n+1}a_n^*=\alpha_n$. It is clear that $A_{n+1}\in\A^{*(n+1)}$ since $B_n[1]\in\A[1]=\A$. In order to finish our proof, we only have to show that $\alpha_{n+1}^Y:\D(Y,A_{n+1})\to M(Y)$ is surjective for all $Y\in\A^{*(n+1)}$. Since this property is obviously inherited by direct summands, it is enough to prove this claim for the objects $Y$ for which there is a triangle \[Z[-1]\stackrel{f}\la X\stackrel{g}\la Y\stackrel{h}\la Z\] with $X\in\A^{*n}$ and $Z\in\A^{*1}$.

Let $y\in M(Y)$, and put $x=M(g)(y)\in M(X)$. Since $X\in\A^{*n}$, the map $\alpha_n^X$ is surjective, so there is $u:X\to A_n$ such that $\alpha_n^X(u)=x$. Further $M(f)\alpha_n^X(u)=M(f)(x)=M(f)M(g)(y)=0$, and by the naturality of $\alpha_n$ this is equivalent to $\alpha_n^{Z[-1]}f_{A_n}(u)=0$, as we can see from the commutative diagram:
\[\diagram
    \D(Z[-1],A_n)\dto_{\alpha_n^{Z[-1]}} &\D(X,A_n)\dto^{\alpha_n^X}\lto_{f_{A_n}}\\
    M(Z[-1]) & M(X)\lto^{M(f)}
\enddiagram\] where $f_*=\D(f,-)$. Thus $uf=f_{A_n}(u)$ belongs to the kernel $K_n(Z[-1])$ of $\alpha_n^{Z[-1]}$. But $Z[-1]\in\A^{*1}[-1]=\A^{*1}$, therefore the sequence $\D(Z[-1],B_n)\to K_n(Z[-1])\to 0$ is exact, what means that there is $v:Z[-1]\to B_n$ such that $uf=b^{Z[-1]}(v)=b_nv$. We get a morphism $w:Y\to A_{n+1}$ completing commutative to a morphism of triangles the diagram:
\[\diagram Z[-1]\rto^{f}\dto^{v}&X\rto^{g}\dto^{u}&Y\rto^{h}\dto^{w}&Z\dto^{v[1]} \\
B_n\rto_{b_n}&A_n\rto_{a_n}&A_{n+1}\rto_{c_n}&B_n[1]
\enddiagram.\]

Putting together all equalities before we get: \[M(g)(y)=x=\alpha_n^X(u)=\alpha_{n+1}^Xa_n^X(u)=\alpha_{n+1}^Xg_{A_{n+1}}(w).\] Using the naturality of $\alpha_{n+1}$ we continue:
\[M(g)(y)=M(g)\alpha_{n+1}^Y(w),\]
therefore $M(g)(y-\alpha_{n+1}^Y(w))=0$.
Furthermore, we know that the sequence \[M(Z)\stackrel{M(h)}\la M(Y)\stackrel{M(g)}\la M(X)\] is exact, so there is $z\in M(Z)$ such that $M(h)(z)=y-\alpha_{n+1}^Y(w)$. But $Z\in\A\subseteq\A^{*n}$, therefore $\alpha_n^Z:\D(Z,A_n)\to M(Z)$ is surjective, therefore we get a map $s:Z\to A_n$ such that $\alpha_n^Z(s)=z$. Using the naturality of $\alpha_{n+1}$ again we get
\[M(h)(z)=M(h)\alpha_n^Z(s)=M(h)\alpha_{n+1}^Za_n^Z(s)=\alpha_{n+1}^Yh_{A_{n+1}}a_n^Z(s)=\alpha_{n+1}^Y(a_nsh).\] Finally we get \[y=\alpha_{n+1}^Y(a_nsh+w),\] which proves our claim.
\end{proof}

\begin{thm}\label{thm:main}
Let $M:\D\to\ModR$ be a cohomological functor, and let $\A\subseteq\D$ be a (full) subcategory such that $\A[1]=\A$, $\A$ is precovering, and $\A^{*n}=\D$ for some $n\geq1$.  Suppose that $M$ is $\A$-finitely generated and, for any $D\in\D$, the kernel of every natural transformation $\D(-,D)\to M$ is also $\A$-finitely generated. Then $M$ is isomorphic to a direct summand of a representable functor. In particular, if $\D$ is karoubian, then $M$ is representable.
\end{thm}

\begin{proof}
 Let $n\geq 1$ for which $\A^{*n}=\D$. We know by Lemma \ref{l:HisA*n-fg} that $M$ is finitely generated, that is, there is $D\in\D$ and an epimorphism $\alpha:\D(-,D)\to M$. Denote by $N$ the kernel of $\alpha$. Then $N$ is cohomological and, by the hypothesis, $\A$-finitely generated. Therefore, there is a natural transformation $\beta:\D(-,C)\to N$, such that $\beta^X$ is surjective for every $X\in\A$. Observe that $N$ is a subfunctor of a representable functor, allowing us to apply Proposition \ref{p:hatH-rep} b). Therefore, the kernel $K$ of $\beta$ is finitely generated, that is, there is a natural epimorphism $\D(-,B)\to K$. Precomposing it with the map $\D(-,A)\to\D(-,B)$, where $A\to B$ is an $\A$-precover, we deduce that $K$ is also $\A$-finitely generated. This allows us to apply Lemma \ref{l:HisA*n-fg} again, to show that $N$ is finitely generated. Together, this means $M$ is finitely presented, and the conclusion follows by Corollary \ref{c:rep-iff-fp}.
\end{proof}

Observe that the shift functor on $\D$ induces an equivalence at the level of the category of $\D$-modules. To be precise, let $M:\D\to\ModR$ be a contravariant functor, that is, a $\D$-module. Then $\Sigma M:\D\to\ModR$ is defined by $\Sigma M(X)=M(X[-1])$. Put $M^*=\bigoplus_{n\in\Z}\Sigma^nM$. 
Fix now an object $G\in\D$. It is clear that $M*(G)$ is an $R$-module. Moreover, denote by $\End^*(G)=\bigoplus_{n\in\Z}\D(G[-n],G)$ the graded endomorphism ring of $G$. It follows immediately that $M^*(G)=\bigoplus_{n\in\Z}M(G[-n])$ becomes a graded left $\End^*(G)$-module. 

Recall that a {\em thick} subcategory of $\D$ is a triangulated subcategory which is closed under direct summands.
Consider a set of objects $\G\in\D$. We denote $\G[\Z]=\{G[n]\mid G\in\G, n\in\Z\}$. Starting with $\G$  we construct inductively $\thick_n(\G)$ as follows:
$\thick_1(\G)=\add(\G[\Z])$, and \begin{align*}
    \thick_{n+1}(\G)=\add(&\{Y\in\D\mid\hbox{there is a triangle }X\to Y\to Z\rightsquigarrow\\ &\hbox{ with }X\in\thick_n(\G)\hbox{ and } Z\in\thick_1(\G)\}).
\end{align*}
It is clear that $\thick(\G)=\bigcup_{n\geq1}\thick_n(\G)$ is the smallest thick subcategory of $\D$ containing $\G$.

A case of particular importance is when $\G$ is a singleton, that is, $\G=\{G\}$, with $G\in\D$, when we simply write $G[\Z]$ instead of $\{G\}[\Z]$. The object $G$ is called a {\em classical generator}, respectively, a {\em strong generator} for $\D$ provided that $\D=\thick(G)$, respectively $\D=\thick_{n}(G)$ for some $n\geq1$. A functor $M:\D\to\ModR$ is called $G$-finite if the $R$-module $M^*(G)$ is finitely generated (in particular, this implies $M(G[n])=0$ for $|n|\gg0$). 

The following gives a new proof for \cite[Theorem 4.2]{N18} or \cite[Corollary 4.18]{R08}:

\begin{cor}\label{c:rep-4-small} Assume $R$ is noetherian and let $G$ be a strong generator for $\D$, such that the functor $\D(-,G)$ is $G$-finite.  Then  a contravariant functor $M:\D\to\ModR$ is isomorphic to a direct summand of a representable functor \iff  it is cohomological and $G$-finite. If further $\D$ is karoubian, then $M$ is representable \iff it is cohomological and $G$-finite. 
\end{cor}

\begin{proof}
   It is clear that $\A=\thick_1(G)$ is stable under shifts and, by hypothesis, $\A^{*n}=\D$ for some $n\geq1$. The argument used in \cite[Lemma 4.1]{N18} to show that $M$ is $\A$-finitely generated (the base case of the induction there) goes as follows: Let $m\geq0$ be such that $M(G[n])=0$, for $|n|>m$. Let $\{a_{n,i}\mid 1\leq i\leq k_n\}$ be a (finite) set of generators for $M(G[n])$ for $|n|\leq m$. By Yoneda we get a set of natural transformations $\alpha_{n,i}:\D(-,G[n])\to M$, with $1\leq i\leq k_n$. For $A=\bigoplus_{|n|\leq m}\bigoplus_{i=1}^{k_n}G[n]$, we get a natural transformation $\alpha:\D(-,A)\to M$ which is an epimorphism when it is restricted to $\A$. Since the functor $\D(-,G)$ is $G$-finite, it follows immediately that $\D(-,D)$ is $G$-finite cohomological functor for all $D\in\thick(G)=\D$. The argument used above for $M$, applied this time for $\D(-,D)$, with $D\in\D$ arbitrary chosen, shows that $\A$ is precovering. If $K$ is the kernel of a natural transformation $\D(-,D)\to M$, then the same argument works again to show that $K$ is $\A$-finitely generated. Therefore, Theorem \ref{thm:main} provides the conclusion. 
\end{proof}

The following Corollary is a variant of \cite[Theorem 4.20]{R08}.
One advantage of the present proof is that it can be directly dualized in order to find a criterion for Brown representability for the dual. Actually, the next Corollary is exactly the dual of \cite[Theorem 1.11]{N09} (see also \cite[Remark 1.12]{N09}), but the proofs are different.   

\begin{cor}\label{c:rep-4-big}
    Assume $\D$ has coproducts and $M:\D\to\ModR$ is a contravariant functor. Suppose that there is a set $\G$ of objects in $\D$ such that $\D=\A^{*n}$, where $\A=\Add\G[\Z]$. Then $M$ is representable \iff $M$ is cohomological and sends coproducts into products.  
\end{cor}

\begin{proof}
    Representable functors are cohomological and send coproducts into products. Conversely, 
    $\A$ is precovering, and taking $A=\bigoplus_{G\in\G[\Z]}\left(\bigoplus_{x\in M(G)}G_x\right)$, where $G_x\cong G$, for all $x\in M(G)$, we get an epimorphism $\D(-,A)\to M$, therefore $M$ is $\A$-finitely generated. Moreover, if $K$ is the kernel of a natural transformation $\D(-,D)\to M$, then $K$ is cohomological and sends coproducts into products. By the same argument $K$ is also $\A$-finitely generated, hence Theorem \ref{thm:main} applies.  
\end{proof}

Observe that the shift functor on $\D$ induces an equivalence at the level of the category of $\D$-modules. To be precise, let $M:\D\to\ModR$ be a contravariant functor, that is, a $\D$-module. Then $\Sigma M:\D\to\ModR$ is defined by $\Sigma M(X)=M(X[-1])$. Put $M^*=\bigoplus_{n\in\Z}\Sigma^nM$.
Fix now an object $G\in\D$.
Denote by $\End^*(G)=\bigoplus_{n\in\Z}\D(G[-n],G)$ the graded endomorphism ring of $G$. It follows immediately that $M^*(G)=\bigoplus_{n\in\Z}M(G[-n])$ becomes a graded left $\End^*(G)$-module.  

\begin{lemma}\cite[Proposition 4.8]{R08}\label{l:fg-rouq}
With the notations above, we have: 
\begin{enumerate}[{\rm a)}] 
\item If $M^*(G)$ is finitely generated as $\End^*(G)$-module, then $M$ is $\thick_1(G)$-finitely generated.
\item If $M$ is $\thick(G)$-finitely generated, then $M^*(G)$ is a finitely generated $\End^*(G)$-module. 
\end{enumerate} 
\end{lemma}

 Note that \cite[Proposition 4.8]{R08} says more, namely that $M^*(G)$ is finitely generated as $\End^*(G)$-module if and only if $M$ is $\thick(G)$-finitely generated. But we cannot follow the whole argument, so we recorded above only the implications we are sure (and actually we need).  



\begin{cor}\label{c:gr-fg} Let $\D$ be a karoubian triangulated category, let $M:\D\to\ModR$ be a cohomological functor, and let $G\in\D$.  Suppose that $M^*(G)$ is finitely generated as $\End^*(G)$-module and $\End^*(G)$ is noetherian. Then both $M$ and the kernel of any natural transformation $\D(-,D)\to M$, with $D\in\D$, are  $\thick_1(G)$-finitely generated. If, in addition, $G\in\D$ is a strong generator, then $M$ is representable.
\end{cor}

\begin{proof} If $D\in\thick_1(G)$, then $\D(-,D)$ is clearly $\End^*(G)$-finitely generated, therefore the noetherian assumption on $\End^*(G)$ implies that the kernel of any natural transformation $\D(-,D)\to M$ has the same finiteness property. We apply twice Lemma \ref{l:fg-rouq} and we deduce the first conclusion. If $G$ is a strong generator, the second conclusion follows  by Theorem \ref{thm:main}.
\end{proof}

Recall that an algebra $A$ over $R$ is called {\em regular} if every finitely generated module is of finite projective dimension. 

\begin{cor}\label{c:noeth-alg}
Let $A$ be a noetherian algebra over $R$ and let $\D=\Dep A$ be the derived category of perfect complexes of modules over $A$. Suppose that every finitely generated $A$-module is of finite projective dimension.  Then every cohomological functor $M:\D\to\ModR$ for which $M^*(A)$ is $A$-finitely generated is representable. 
\end{cor}

\begin{proof} If $A$ satisfies the assumptions of the hypothesis, it is a strong generator for $\Dep A$, according to \cite[Theorem 7.25]{R08}. If is clear that $E\cong A$ is noetherian and Corollary \ref{c:gr-fg} applies.
\end{proof}

\section{Pure dimension less than or equal to 1}

From now on $\D$ will be a triangulated category with arbitrary coproducts. Recall that an object $C\in\D$ is called compact if the covariant functor $\D(C,-):\D\to\ModR$ commutes with coproducts. Denote by $\D^c$ the subcategory of all compact objects in $\D$. Suppose $\D$ is compactly generated, that is, $\D^c$ is essentially small and $\D(C,X)=0$ for all $C\in\D^c$ implies $X=0$.
A main tool in our consideration is the restricted Yoneda functor $\yon:\D\to\Mod{\D^c}$ given by $\yon X=\D(-,X)|_{\D^c}$. Thus, if we know that $\D^c$ is essentially small, then $\D$ is compactly generated exactly if the functor $\yon$ reflects zero objects. Since a map $f:X\to Y$ in $\D$ is an isomorphism if and only if the third term of its completion to a triangle vanishes, it follows further that this is equivalent to $\yon$ reflecting isomorphisms, see \cite[Lemma 6.2.8]{N01}.  Recall that the kernel ideal of $\yon$ is the ideal of {\em phantom maps}, that is, a map $\phi$ in $\D$ is phantom exactly if $\yon(\phi)=0$.

\begin{lemma}\label{l:pure-tri}\cite[Lemma 1.3]{K00}  For a triangle \[X\stackrel{\alpha}\la Y\stackrel{\beta}\la Z\stackrel{\phi}\la X[1]\] in $\D$, the following are equivalent:
\begin{enumerate}[{\rm (i)}]
    \item $\yon(\alpha)$ is a monomorphism.
    \item $\yon(\beta)$ is an epimorphism.
    \item $\yon(\phi)=0$, that is $\phi$ is a phantom map.
    \item The sequence $0\to\yon X\to\yon Y\to\yon Z\to 0$ is short exact in $\Mod{\D^c}$.
\end{enumerate}
\end{lemma}

We call {\em pure} triangles satisfying the equivalent conditions in Lemma \ref{l:pure-tri} above.  An object $X\in\D$ is called pure projective (injective) if $\D(X,-)$ (respectively, $\D(-,X)$ sends pure triangles into short exact sequences in $\ModK$. By \cite[Corollary 4.5]{CS98} or \cite[Lemma 8.1]{Be00}, the class of pure projective objects in $\D$ coincides with $\Add({\D^c})$. It is clear that this class is precovering, in the sense that for every $X\in\D$ there is an object $P_X\in\Add(\D^c)$ and a map $P_X\to X$ such that every map $P\to X$, with $P\in\Add(\D^c)$ factors through the previous one. We say that $X\in\D$ is of pure projective dimension less than or equal to 1, and we write $\ppdim X\leq1$ if there is a pure triangle
$P_1\to P_0\to X\rightsquigarrow$ with $P_0,P_1$ being pure projective. We denote by $\D_d$ the full subcategory of $\D$ consisting of the objects $X\in\D$ with $\ppdim X\leq1$. 

\begin{prop}\label{p:ppdim}
The following are equivalent for an objects $X\in\D$:
\begin{enumerate}[{\rm (i)}]
    \item $\ppdim X\leq1$.
    \item $\yon_{X,Y}:\D(X,Y)\to\Hom_{\D^c}(\yon X,\yon Y)$ is surjective, for all $Y\in\D$.
\end{enumerate}
In particular, the restriction of $\yon$ to $\D_{d}\to\Mod{\D^c}$ is full and its essential image consists of those ${\D^c}$-modules of projective dimension less than or equal to 1.
\end{prop}

\begin{proof}
First note that if $P$ is pure projective, then $\yon_{P,Y}$ is bijective for all $Y\in\D$. Indeed this property are inherited by direct summands, and for a coproducts of compact objects $P\cong\coprod_{i\in\I}C_i$, this follows from Yoneda lemma since
    \begin{align*}
      \Hom_{\D^c}\left(\yon(\coprod_{i\in I}C_i),\yon Y\right)&\cong\Hom_{\D^c}\left(\coprod_{i\in I}\yon(C_i),\yon Y\right)\cong\prod_{i\in I}\Hom_{\D^c}(\yon C_i,\yon Y)\\
      &\cong\prod_{i\in I}\D(C_i,Y)\cong\D\left(\coprod_{i\in I}C_i,Y\right).
    \end{align*}
   (i)$\Rightarrow$(ii).  Let $X\in\D$ be an object with $\ppdim X\leq1$. Then there is a pure triangle $P_1\to P_0\to X\rightsquigarrow$, with $P_1,P_0\in\Add(\D^c)$. For $Y\in\D$, the commutative diagram with exact rows
\[\diagram &\D(X,Y)\rto\dto_{\yon_{X,Y}}&\D(P_0,Y)\rto\dto_{\yon_{P_0,Y}}^{\cong}&\D(P_1,Y)\dto_{\yon_{P_1,Y}}^{\cong}\\
0\rto&\Hom_{\D^c}(\yon X,\yon Y)\rto&\Hom_{\D^c}(\yon P_0,\yon Y)\rto&\Hom_{\D^c}(\yon P_1,\yon Y)
\enddiagram\]
shows that $\yon_{X,Y}$ is surjective.

(ii)$\Rightarrow$(i). Consider a projective presentation $\yon P_1\to \yon P_0\to \yon X\to 0$ of $\yon X$ in $\Mod{\D^c}$. Since $\yon_{P,-}$ is bijective for all $P\in\Add(\D^c)$, there are unique maps $P_1\stackrel{\alpha}\la P_0\stackrel{\beta}\la X$ that induce the above presentation. The composition $\beta\alpha$ vanishes, therefore completing $\alpha$ and $\beta$ to triangles, we get a  morphism of triangles in $\D$:
\[\diagram
P_1\rto^{\alpha}\dto& P_0\rto\ddouble& Y\rto\dto^{\gamma}& P_1[1]\dto\\
X'\rto& P_0\rto^{\beta} & X\rto &X'[1]
\enddiagram\]

Observe that $\yon(\beta)$ is surjective, therefore the lower triangle is pure. On the other hand, from the commutative diagram with exact rows in $\Mod{\D^c}$:
\[\diagram
    \yon P_1\rto\ddouble&\yon P_0\rto\ddouble& \yon X\rto& 0\\
    \yon P_1\rto&\yon P_0\rto&\yon Y&
    \enddiagram\]
we get a map $\yon X\to\yon Y$ that completes it commutatively. By hypothesis, it must be of the form $\yon(\delta)$ for some $\delta:X\to Y$ in $\D$. It follows that $\yon(\gamma\delta)$ is an isomorphism, therefore, so is $\gamma\delta$ too. This implies that the above morphism of triangles splits, so $X'$ is isomorphic to a direct summand of $P_1$. This shows that $X'$ is pure projective, which precisely means that $\ppdim X\leq1$.

Finally, it is clear that the functor $\D_d\stackrel{\subseteq}\la\D\stackrel{\yon}\la\Mod{\D^c}$ is full and for every $X\in\D_d$, we have $\pdim\yon X\leq1$. In contrast, if $M\in\Mod{\D^c}$ has $\pdim M\leq 1$, then there is a short exact sequence $0\to\yon P_1\to\yon P_0\to M\to 0$ in $\Mod{\D^c}$, with $P_1,P_0\in\Add(\D^c)$. The map $\yon P_1\to\yon P_0$ is of the form $\yon(\alpha)$ for a unique $\alpha:P_1\to P_0$ in $\D$. Since $\yon(\alpha)$ is a monomorphism, it follows that the triangle that completes it $P_1\stackrel{\alpha}\la P_0\to X\rightsquigarrow$ is pure, therefore, $\ppdim X\leq 1$ and $M\cong\yon X$.
 \end{proof}

\begin{cor}\label{c:ppdim-for-hocolim}
    Let $X\in\D$ such that $X\cong\hocolim_{n\geq0}P_n$, with $P_n\in\Add{\D^c}$. Then $X\in\D_d$ and, consequently, $\yon_{X,Y}$ is surjective for all $Y\in\D$.
\end{cor}

\begin{proof}
    Note that the triangle defining the homotopy colimit (up to a non-canonical isomorphism): \[\coprod_{n\geq0}P_n\stackrel{1-shift}\la\coprod_{n\geq0}P_n\la X\rightsquigarrow\] is pure according to \cite[Lemma 2.8]{N96}. Therefore, $\ppdim X\leq1$ and consequently, Proposition \ref{p:ppdim} applies.
\end{proof}

\section{Locally coherent cohomological functors}

  Fix a module $M\in\Mod{\D^c}$, and an object $G\in\D^c$. We will say that $M$ is $G$-{\em locally coherent}, provided that the following two conditions hold:
 \begin{enumerate}[\rm (LC1)]
  \item  There are $D\in\D^c$ and a natural transformation $\alpha:\yon D\to M$ such that $\alpha^G:\D^c(G,D)\to M(G)$ is surjective.
  \item If $\beta:\yon D\to M$ is a natural transformation, then there is a map $f:C\to D$ in $\D^c$ such that $\beta\yon f=0$ and \[\D^c(G,C)\to\D^c(G,D)\to M(G)\]
is exact.  
\end{enumerate} 
If $\G\subseteq\D^c$ is a set of compact objects, we say that $M$ is $\G$-{\em locally coherent}, provided that it is $G$-locally coherent for every $G\in\G$. 
We will say that an object $X\in\D$ has $G$-{\em locally coherent cohomology} if $\yon X\in\Mod{\D^c}$ is
$G[\Z]$-locally coherent.

Note that conditions (LC1) and (LC2) above are inspired by conditions (a) and (b) characterizing locally finitely presented functors in \cite[pag. 205]{R08}. Sure, we regard the category $\D^c$ as a subcategory in the larger $\D$ but this difference is not essential. More important is that the requirement in (LC1) and (LC2) is only about $\alpha^G$, respectively $\beta^G$, while \cite{R08} the conclusion is about $\alpha^{G[n]}$ and $\beta^{G[n]}$, for all $n\in\Z$. Even if we consider a $G[\Z]$-locally coherent module (i. e. functor) $M$, it seems that our notion is weaker than that of a locally finitely presented functor in \cite{R08}, since the natural transformations $\alpha$ and $\beta$ depend here on $n\in\Z$, while in \cite{R08} they are unique. 
However, some arguments can be used, as observed in \cite[Remark 4.7]{R08}. 

\begin{lemma}\label{l:lcoh} Let $M\in\Mod\D^c$ and $G\in\D^c$. Then: 
\begin{enumerate}[\rm (a)]
    \item $\D^c\subseteq\D_G$.
    \item $M$ is $G$-locally coherent \iff $M$ satisfies (LC1) and the kernel of any natural transformation $\yon C\to M$ satisfies (LC1).
   \item If $\beta_i:\yon D_i\to M$, with $i=1,2$ are so that (LC2) holds for $\beta=\beta_1+\beta_2:\yon(D_1\oplus D_2)\to M$, then (LC2) holds for $\beta_i$ ($i=1,2$). 
   \item Assume (LC1) holds. If (LC2) holds for those $\beta:\yon D\to M$, with $D\in\D^c$, for which $\beta^{G}:\D^c(G,D)\to M(G)$ is surjective, then (LC2) holds for all $\beta$. 
   \item If $M$ is $G$-locally coherent, then $M$ is $X$-locally coherent for any direct summand $X$ of $G$. 
\end{enumerate}
Suppose in addition that $M$ is cohomological; then 
\begin{enumerate}[\rm (f)]
  \item If $X_1\to X\to X_2\rightsquigarrow$ is a triangle and $M$ is $X_i$-locally coherent, for $i=1,2$, then $M$ is $X$-locally coherent. 
  \item[\rm (g)] If $M$ is $G[\Z]$-locally coherent then $M$ is $\thick(G)$-locally coherent. 
\end{enumerate}
\end{lemma}

\begin{proof}
For (a) use the proof of \cite[lemma 4.2]{R08}. 
Statement (b) is obvious, and statement (e) follows immediately. For (c) and (d) use the proof of \cite[Lemma 4.4]{R08}. Statement (f) uses the same argument as \cite[Lemma 4.6]{R08} and statement (g) follows from (e) and (f). 
\end{proof}

Note that if $G\in\D^c$ is a single generator for $\D$ in the sense that $X\in\D$ such that $\D(G[i],X)=0$ implies $X=0$, then $\thick(G)=\D^c$. In this case,  we denote 
\[\D_{lcoh}=\{X\in\D\mid X\hbox{ has }G\hbox{-locally coherent cohomology}\}\]  
and Lemma \ref{l:lcoh} (f) implies that $\D_{lcoh}$ does not really depend on $G$ since 
\[\D_{lcoh}=\{X\in\D\mid X\hbox{ has }D\hbox{-locally coherent cohomology, for any }D\in\D^c\}.\]  

\begin{lemma}\label{l:approx}
Let $M\in\Mod{\D^c}$ be a cohomological module and $G\in\D^c$, such that $M$ is $G[\Z]$-locally coherent. Then there is a chain of compact objects $D_0\to D_1\to D_2\to\ldots$ together with natural transformations $\varphi_n:\yon D_n\to M$ compatible with the induced maps $\yon D_n\to\yon D_{n+1}$ such that the natural transformation $\colim\yon D_n\to M$ evaluated at $G[k]$ is invertible for all $k\in\Z$. Furthermore, if we denote $D=\hocolim D_n$, then we get a natural isomorphism
\[\D(X,D)\to M(X),\hbox{ for all }X\in\thick(G).\]
\end{lemma}

\begin{proof} Note that $M$ is $\thick(G)$-locally coherent, as we have seen in Lemma \ref{l:lcoh} (f). 
   We construct the requested chain by induction. For all $n\geq 0$, we consider the objects $G_n=\coprod_{i=-n}^nG[n]\in\thick(G)$.  For $n=0$, we use (LC1) in order to construct $\varphi_0:\yon D_0\to M$ such that $D_0\in\D^c$ and $\D^c(G,D_0)\to M(G)$ is surjective. 
   For $n\geq0$, suppose that we have constructed $D_n\in\D^c$, and $\varphi_n:\yon D_n\to M$, which is surjective evaluated at $G_n$. Applying condition (LC2) we get a map $f_n:C_n\to D_n$ in $\D^c$ such that $\varphi_n\yon f_n=0$ and the sequence \[\D^c(G_n,C_n)\to\D^c(G_n,D_n)\to M(G_n)\to0\] 
   is exact.  Complete $f_n$ to a triangle $C_n\stackrel{f_n}\to D_n\stackrel{g_n}\to E_n\rightsquigarrow$. Since $\D^c$ is a triangulated subcategory and $C_n,D_n\in\D^c$, we deduce that $E_n\in\D^c$. Since $M$ is cohomological we get an exact sequence $M(E_{n})\to M(D_n)\to M(C_n)$ in $\ModK$, and the Yoneda lemma provides us the natural transformation $\psi_{n}:\yon E_{n}\to M$ making commutative the diagram
   \[\diagram\yon D_n\rto^{\yon g_n}\dto_{\varphi_n}&\yon E_{n}\dto^{\psi_{n}}\\ M\rdouble &M\enddiagram.\] It follows that $\psi_n$ evaluated at $G_n$ is surjective.
   On the other hand, applying condition (LC1) again, we get $D'_n\in\D^c$ and a natural map $\varphi'_n:\yon D'_n\to M$ such that the induced map \[\D(G[-(n+1)]\sqcup
   G[n+1],D_n')\to M(G[-(n+1)]\sqcup G[n+1])\] is surjective. We take $D_{n+1}=E_n\sqcup D'_n$ and \[\varphi_{n+1}=(\psi_n,\varphi'_n):\yon D_{n+1}\cong\yon E_n\oplus\yon D'_n\to M\] It is clear that $\varphi_{n+1}$ evaluated at $G_{n+1}=G_n\sqcup G[-(n+1)]\sqcup G[n+1]$ is surjective.  Moreover, if we denote by $d_n:D_n\to D_{n+1}$ the composite map $D_n\to E_n\to D_{n+1}$, then $\varphi_{n+1}\yon d_n=\psi_n\yon f_n=\varphi_n$. We denote by $\delta_n:\yon D_n\to\colim_{n\geq0}\yon D_n$ the canonical maps of the colimit and get a natural transformation $\varphi:\colim_{n\geq0}\yon D_n\to M$ such that $\varphi\delta_n=\varphi_n$, for all $n\geq0$.

For an integer $k\in\Z$, the object $G[k]$ is by construction a direct summand in $G_{|k|}$, therefore it is enough to show that the map $\colim_{n\geq0}\D^c(G_k,D_n)\to M(G_k)$ is bijective for all $k\geq0$. Fix therefore $k\geq0$. If $y\in M(G_k)$, then since $\varphi_k$ evaluated at $G_k$ is surjective, we get an element $x_k\in\D^c(G_k,D_k)$ such that $\varphi_k(x_k)=y$. Thus, for $x=\delta_k(x_k)\in\colim_{n\geq0}\D^c(G_k,D_n)$ we have $\varphi(x)=y$, showing that the map $\varphi$ is surjective. On the other hand, if $x\in\colim_{n\geq0}\D^c(G_k,D_n)$ is such that $\varphi(x)=0$, then there is $n\geq 0$, which can be chosen to be greater or equal to $k$ and an element $x_n\in\D^c(G_k,D_n)$ such that $x=\delta_n(x_n)$. Then $\varphi_n(x_n)=0$. Since $n\geq k$ we know that $G_k$ is a direct summand of $G_n$, therefore, $\D^c(G_k,C_n)\to\D^c(G_k,D_n)\to M(G_k)$ is exact.  Hence, $x_n$ is in the image of $\yon f_n$, that is, there is $z_n\in\D^c(G_k,C_n)$ such that $\yon f_n(z_n)=x_n$. It follows that $\yon g_n(x_n)=0$, therefore $\yon d_n(x_n)=0$. This shows that $x=\delta_{n+1}(\yon d_n(x_n))=0$, which proves the claim.

Finally, if $D=\hocolim_{n\geq0}D_n$, then we know by \cite[Lemma 5.8]{Be00} that there is an isomorphism $\yon D\stackrel{\sim}\to\colim_{n\geq0}\yon D_n$. The first part of this lemma gives us a natural transformation $\yon D\to M$ which is invertible evaluated at $G[k]$, for all $k\in\Z$. Since every object in $\thick(G)$ can be constructed in a finite number of steps as direct summands of iterated extensions of objects in $G[\Z]$, the isomorphism $\D(X,D)\stackrel{\sim}\to M(X)$ for all $X\in\thick(G)$ follows by the five lemma.

 \end{proof}


\begin{thm}\label{t:BA}
    Let $\D$ be a triangulated category that has a compact generator $G$. Then every $X\in\D_{lcoh}$ is a homotopy colimit of compact objects. Furthermore, the functor \[\D_{lcoh}\stackrel{\subseteq}\la\D\stackrel{\yon}\la\Mod{\D^c}\] is full, and its essential image consists of $\D^c$-locally coherent cohomological modules. Moreover, every locally coherent cohomological $\D^c$-module $M\in\Mod{\D^c}$ is of projective dimension $\leq1$.  
\end{thm}

\begin{proof}
   We use the same symbol $\yon$ for both the initial functor $\D\to\Mod{\D^c}$ and its restriction $\D_{lcoh}\to\Mod{\D^c}$. As we already noted $\D^c=\thick(G)$, since $G$ is a compact generator. By the very definition of $\D_{lcoh}$, for every $X\in\D_{lcoh}$, the cohomological $\D^c$-module $\yon X$ is $G[\Z]$-locally coherent, consequently, it satisfies the hypothesis of Lemma \ref{l:approx}. Then there is a natural isomorphism $\yon D\cong\yon X$, for some $D=\hocolim_{n\geq 0}D_n$, with $D_n\in\D^c$. Furthermore, Corollary \ref{c:ppdim-for-hocolim} implies that this natural transformation comes from a map $D\to X$ in $\D$, which has to be an isomorphism, since $\D$ is compactly generated.  Therefore, the same Corollary \ref{c:ppdim-for-hocolim} implies that $\yon:\D_{lcoh}\to\Mod{\D^c}$ is full. Finally, if $M\in\Mod{\D^c}$ is a locally coherent cohomological $\D^c$-module, then applying Lemma \ref{l:approx} again, we get an isomorphism $\yon D\to M$, for some $D\in\D$. The assumptions in $M$ imply $D\in\D_{lcoh}$. The last statement is now obvious.
\end{proof}


Keep the hypotheses of Theorem \ref{t:BA}, namely $\D$ is a triangulated category with a compact generator $G$. Define:
\[\D^-_{lcoh}=\{X\in\D_{lcoh}\mid\D(G[-n],X)=0\hbox{ for }n\gg 0\},\]
\[\D^+_{lcoh}=\{X\in\D_{lcoh}\mid\D(G[-n],X)=0\hbox{ for }n\ll 0\},\] and $\D^b_{lcoh}=\D^-_{lcoh}\cap\D^+_{lcoh}$.
For brevity, we will introduce the following shorthand. If $M\in\Mod{\D^c}$ is a module, we will say that $M^*(G)$ is $s$-bounded, with $s\in\{-,+,b\}$, if $\Sigma^nM(G)=0$, for $n\gg 0$, $n\ll0$, respectively $|n|\gg0$ (i. e., the graded $R$-module $M^*(G)$ is bounded above, bellow, respectively, two-sided). Therefore, we get $\D^s_{lcoh}=\{X\in\D_{lcoh}\mid\yon X\hbox{ is }s\hbox{-bounded}\}$.

\begin{cor}\label{c:BA} Let $\D$ be a triangulated category that has a compact generator $G$. Then the functor \[\D^s_{lcoh}\stackrel{\subseteq}\la\D\la\Mod{\D^c},\hbox{ with }s\in\{-,+,b\},\] is full and its essential image consists of locally coherent cohomological modules $M\in\Mod{\D^c}$, which are $s$-bounded.
\end{cor}

\begin{proof}
   In all three cases, the functors are full, as they are restrictions of the (full) functor $\D_{lcoh}\stackrel{\subseteq}\la\D\la\Mod{\D^c}$. Furthermore, the description of the essential image follows from the definition of the categories $\D^-_{lcoh}$, $\D^+_{lcoh}$ and $\D^b_{lcoh}$.  
\end{proof}

\begin{lemma}\label{l:epi2mc}
   Let $M\in\Mod{\D^c}$ such that $M(G)$ is a finitely generated $R$-module, for some object $G\in\D^c$. Then there are $D\in\D^c$ and a natural transformation $\yon D\to M$ inducing a surjective map $\D^c(G,D)\to M(G).$ 
\end{lemma}

\begin{proof}
    Since all modules are epimorphic images of free ones, we know that there is a natural epimorphism $\bigoplus_{i\in I}\yon D_i\to M$ and evaluating it at $G$ we obtain an exact sequence in $\Mod R$:
    \[\bigoplus_{i\in I}\D^c(G,D_i)\to M(G)\to 0.\] Since $M(G)$ is finitely generated, there is a finite subset $J\subseteq I$, such that \[\bigoplus_{i\in J}\D^c(G,D_i)\to M(G)\to 0\] remains exact. Therefore, it is enough to take $D=\coprod_{i\in J}D_i\in\D^c$ and the restricted natural transformation $\yon D\to M$ does the job we need.
\end{proof}
 
\begin{lemma}
If $\G\in\D^c$ is a compact generator for $\D$, with the property that $\D(G,G[n])$ is a finitely generated $R$-module for all $n\in\Z$, then the $R$-module $\D(C,D)$ is finitely generated for all $C,D\in\D^c$
\end{lemma}

\begin{proof}
   The proof is an immediate consequence of the fact that $\D^c=\thick(G)$.  
\end{proof}

\begin{cor}\label{c:lcoh} Assume the ring $R$ is noetherian. Suppose that $\D$ has a compact generator $G\in\D^c$ such that $\D(G,G[n])\in\modK$ for all $n\in\Z$. If  $M\in\Mod{\D^c}$ is a cohomological module, then $M$ is $\D^c$-locally coherent \iff $M(C)\in\modK$, for all $C\in\D^c$. Consequently \[\D_{lcoh}=\{X\in\D\mid\D(C,X)\in\modK\hbox{ for all }C\in\D^c\}.\] 
\end{cor}

\begin{proof} Clearly, $M(G[n])\in\modK$ for all $n\in\Z$ \iff $M(C)\in\modK$ for all $C\in\D^c$. 

Let $C\in\D^c$. If $M(C)\in\modK$ and $K$ is the kernel of a natural transformation $\yon D\to M$, with $D\in\D^c$, then the noetherian assumption on $R$ implies that $K(C)\in\modK$. Therefore, $M$ is $\D^c$-locally coherent, by Lemma \ref{l:epi2mc}. Conversely, if $M$ is $\D^c$-locally coherent, then $M(C)$ is a quotient of $\D(C,D)$, hence it belongs to $\modK$.   
\end{proof}

\section{Application: the case of derived categories over dg-algebras}

In this Section we will particularize our general results for the case of derived categories of dg-algebras. Recall that a dg-algebra over $R$ is a $\Z$-graded algebra $A=\bigoplus_{n\in\Z}A^n$, together with a differential $d:A\to A$, which is a $R$-linear map, homogeneous of degree 1, satisfying $d^2=0$ and the graded Leibnitz rule, namely $d(ab)=d(a)b+(-1)^nad(b)$, for all $a\in A^n$ and all $b\in A$. We follow \cite{K94} in the construction of the following categories: The category of all dg-$A$-modules, the homotopy category of $A$ and the derived category of $A$; this last category is denoted $\Der A$, is compactly generated with $A$ a single compact generator. The subcategory of compact objects in $\Der A$ consists of the so-called perfect complexes and is denoted by $\Dep A$. Note that perfect complexes can be described as direct summands of an object in $\Der A$, which is represented by a differential graded module, having a finite filtration by differential graded submodules, whose quotients are finite direct sums of shifts of $A$ (see \cite[Tag 09R3]{stacks}).  Recall also that the dg-algebra $A$ is called {\em proper}, respectively {\em regular}, if $\bigoplus_{n\in\Z}H^nA\in\modK$, respectively $A$ is a strong generator of $\Dep A$, see \cite[Definition 1.5]{O23}.  

\begin{prop}\label{p:B4dg}
Let $R$ be a noetherian ring, and let $A$ be a proper and regular dg-algebra over $R$, satisfying also the property that $H^n(A)\in\modK$, for all $n\in\Z$. Then a module $M\in\Mod{\Dep{A}}$ is representable if and only if $M$ is cohomological and \[\bigoplus_{n\in\Z}\Sigma^nM(A)\in\modK\] (that is, $M$ is $A$-finite). 
\end{prop}

\begin{proof}
Since $A$ is regular, $\Dep A$ is strongly generated by $A$. Moreover, the properness can be reformulated by saying that the module $\Dep A(-,A)$ is $A$-finite.  Then the conclusion here follows from Corollary \ref{c:rep-4-small}. 
\end{proof}

Furthermore, we consider the subcategory:
\[\Dec A=\{X\in\Der A\mid H^i(X)\in\modK\hbox{ for all }i\in\Z\}.\]
A priori, we did not impose any finiteness conditions on $A$.  Even in this generality, we can define
\[\Decm A=\{H\in\Dec A\mid H^n(X)=0\hbox{ for }n\gg0\},\] 
\[\mathbf{D}^+_{coh}(A)=\{H\in\Dec A\mid H^n(X)=0\hbox{ for }n\ll0\},\] and
\[\Derb A=\{H\in\Dec A\mid H^n(X)=0\hbox{ for }|n|\gg0\}.\] 

We observe that, by Corollary \ref{c:lcoh}, $M\in\Mod{\Dep A}$ is $\Dep A$-locally coherent \iff $\Sigma^n M(A)\in\modK$, for all $n\in\Z$. Furthermore, this is equivalent to $M:\Dep A\to\modK$; we will simply call {\em locally coherent} a functor that satisfies this property. 

\begin{prop}\label{p:BA4dg}
   Let $R$ be a noetherian ring, and let $A$ be a dg-algebra over $R$, with the property that $H^n(A)\in\modK$, for all $n\in\Z$. Then the following statements hold: 
   \begin{enumerate}[{\rm (a)}] 
        \item The functor \[\Dec A\stackrel{\subseteq}\la\Der A\la\Mod{\Dep A}\] is full, and its essential image consists of locally coherent cohomological modules. 
        \item $\ppdim X\leq 1$ holds for every $X\in\Dec A$. 
        \item For $s\in\{-,+,b\}$, the functor \[\mathbf{D}^s_{coh}(A)\stackrel{\subseteq}\la\Der{A}\la\Mod{\Dep{A}}\] is full and its essential image consists of locally coherent cohomological modules $M\in\Mod{\Dep{A}}$, which are $s$-bounded.  
    \end{enumerate}
\end{prop}

\begin{proof} 
Statements (a) and (b) follow from Theorem \ref{t:BA} and statement (c) follows from Corollary \ref{c:BA}. To be more precise, put $\D=\Der A$. Then, it is enough to observe that for a complex $X\in\Der R$, we have $H^n(X)\cong\Der A(A[-n],X)$, therefore $\D_{lcoh}=\Dec A$, according to Corollary \ref{c:lcoh}. Analogous identifications can be made for $\Decm A$, $\mathbf{D}^+_{coh}(A)$, and $\Derb A$. 
\end{proof}

It is sometimes useful to assume that $A$ is non-positive, that is, $A^n=0$, for $n>0$. If this is the case, then we know by \cite[Remark 3.3]{N18}, that the category $\Der A$ is approximable in the sense of \cite[Definition 0.21]{N18}. Moreover, the pair $(\D^{\leq0},\D^{\geq0})$
is a t-structure, where
\[\D^{\leq0}=\{X\in\Der A\mid H^n(X)=0\hbox{ for }n>0\},\] \[\D^{\geq0}=\{X\in\Der A\mid H^n(X)=0\hbox{ for }n<0\}.\] This t-structure is generated by the compact generator $A$ of $\Der A$ in the sense that $\D^{\leq0}=\{X\in\Der A\mid \Der A(A,X[n])=0\hbox{ for }n>0\}$; in particular, this t-structure is in the preferred equivalence class, defined in \cite[Definition 0.14]{N18}. As usual, denote $\D^{\leq n}=\D^{\leq 0}[-n]$ and $\D^{\geq n}=\D^{\geq 0}[-n]$, for all $n\in\Z$. Consider further $\D^-=\bigcup_{n\in\Z}\D^{\leq n}$ and $\D^+=\bigcup_{n\in\Z}\D^{\geq n}$. Therefore, we get $\Decm A=\D^-\cap\Dec A$ and $\Derb A=\D^-\cap\D^+\cap\Dec A$. 

\begin{prop}\label{p:ff4dg}
   Let $R$ be a noetherian ring, and let $A$ be a non-positive dg-algebra over $R$, with the property that $H^n(A)\in\modK$, for all $n\in\Z$. Then the functor \[\Derb A\stackrel{\subseteq}\la\Der{A}\la\Mod{\Dep{A}}\] is not only full, but also faithful.
\end{prop}

\begin{proof} The faithfulness of this functor follows from \cite[Theorem 0.3]{N18}, since in this case the category $\Der A$ is approximable.
\end{proof}

We end this paper with some considerations about the way our results apply for schemes.  
We consider a scheme $\mathbb{X}$ that is quasi-compact and quasi-separated. Then it is known that the category $\Qcoh\mathbb{X}$ of quasi-coherent $\mathbb{X}$-modules is Grothendieck, and the category $\Derq{\mathbb{X}}$ of $\mathbb{X}$-modules with quasi-coherent cohomology is compactly generated and has a single compact generator $G$. Moreover, compact objects in $\Derq{\mathbb{X}}$ are precisely perfect complexes; recall that a complex is perfect means that it is locally quasi-isomorphic a bounded complex of vector bundles; see \cite[Theorem 3.1.1]{BvdB03}. We denote by $\Dep{\mathbb{X}}=\Derq{\mathbb{X}}^c$.  
According to \cite[Corollary 3.1.8]{BvdB03}, 
$\Derq{\mathbb{X}}$ is equivalent to the derived category of dg-algebra $A=A(\mathbb{X})$ with bounded cohomology. Using this observation, we can apply Proposition \ref{p:BA4dg} for the category $\Derq{\mathbb{X}}$. To be more precise, the equivalence
\[\Derq{\mathbb{X}}\to\Der A\] sends the compact generator $G$ to $A\in\Der A$. Using this observation we can see that  \begin{align*} \Derq{\mathbb{X}}_{lcoh}&=\{X\in\Derq{\mathbb{X}}\mid\D(C,X)\in\modK\hbox{ for all }C\in\Dep{\mathbb{X}}\}\\ &=\{X\in\Derq{\mathbb{X}}\mid\D(G[n],X)\in\modK\hbox{ for all }n\in\Z\}\end{align*} is exactly the image of $\Dec A$ under converse $\Der A\to\Derq{\mathbb{X}}$ of the above equivalence. We don't know, but it seems to be very interesting, if $\Derq{\mathbb{X}}_{lcoh}$ can be characterized cohomologically in the same vein as $\Dec A$ can. We also note that $\Derq{\mathbb{X}}^-_{lcoh}$, $\Derq{\mathbb{X}}^+_{lcoh}$, and $\Derq{\mathbb{X}}^b_{lcoh}$ are images of $\Decm A$, $\mathbf{D}^+_{coh}(A)$, respectively $\Derb A$. We also record the property stated in \cite[Proposition 6.9]{R08}: For $X\in\Derq{\mathbb{X}}$, the $\Dep{\mathbb{X}}$-module $(\yon X)^*(G)$ is $s$-bounded, with $s\in\{-,+,b\}$, \iff $H^n(X)=0$ for $n\gg0$, respectively, $n\ll0$, or $|n|\gg0$, consequently the boundedness property can be characterized cohomologically. If $\mathbb{X}$ is quasi-compact and separated, then the category $\Derq{\mathbb{X}}$ is equivalent to $\Der{\Qcoh\mathbb{X}}$ and it is approximable, see \cite[Example 3.6]{N18}. Therefore, \cite[Theorem 0.3]{N18} applies in this case in order to construct a fully faithful functor similar to the one in Proposition \ref{p:ff4dg}, even if we are not sure that the corresponding dg-algebra $A(\mathbb{X})$ is non-positive. However, to avoid confusion, we have to point out that the domain of this functor in \cite[Theorem 0.3]{N18} is not generally known to coincide with $\Derb{\mathbb{X}}$. In order to have this equality, we have to assume more restrictive conditions on $\mathbb{X}$, e. g., to be noetherian. If $X$ is proper over the noetherian ring $R$, then the condition $H^n(A)\in\modK$ is automatically satisfied for $A=A(\mathbb{X})$.

Assume now that $\mathbb{X}$ is projective over the noetherian ring $R$, that is, $\mathbb{X}$ is a closed subscheme of the projectivization of a vector bundle over $\mathrm{Spec}(R)$ (see \cite[Definition 1.1]{B22}). Then by \cite[Theorem 1.5]{B22} we know 
\[\Derq{\mathbb{X}}_{lcoh}=\{X\in\Derq{\mathbb{X}}\mid H^n(X)\in\modK\hbox{, for all }n\in\Z\},\]
that is, $\Derq{\mathbb{X}}_{lcoh}=\Dec{\mathbb{X}}$ can be characterized cohomologically in this case.  Therefore, Proposition \ref{p:BA4dg} can be stated in this case simply by replacing $A$ everywhere with $\mathbb{X}$. Furthermore, an analogous statement with Proposition \ref{p:ff4dg} holds true, but with the same warning as before, the involved dg-algebra is not necessary known to be non-positive.

\end{document}